\documentclass[12pt]{amsart}
\usepackage{txfonts}
\usepackage{amsfonts}
\usepackage{amsfonts}
\usepackage{amssymb,latexsym}
\usepackage{cite}
\usepackage{enumerate}
\newtheorem{theorem}{Theorem}

\newtheorem{remark}[theorem]{Remark}
\newtheorem{proposition}[theorem]{Proposition}
\newtheorem{lemma}[theorem]{Lemma}

\numberwithin{equation}{section}
\numberwithin{theorem}{section}

\usepackage{titletoc}

\usepackage{hyperref}
\hypersetup{hypertex=true,colorlinks=true,linkcolor=blue,anchorcolor=blue,citecolor=blue}
\begin{document}
\title[Overdetermined problems]{Serrin-type Overdetermined problems for Hessian quotient equations}
\author{Zhenghuan Gao}
\address{Department of Mathematics, Shanghai University, Shanghai, 200444, China}
\email{gzh@shu.edu.cn}
\author{Xiaohan Jia} 
\address{School of Mathematical Sciences, Xiamen University, Xiamen, 361005, China}
\email{jiaxiaohan@xmu.edu.cn}
\author{Dekai Zhang}  
\address{Department of Mathematics, Shanghai University, Shanghai, 200444, China}
\email{dkzhang@shu.edu.cn}

\begin{abstract}
We prove the symmetry of solutions to overdetermined problems for a class of fully nonlinear equations, namely Hessian quotient equations and Hessian quotient curvature equations. Our approach is based on establishing a Rellich-Pohozaev type identity for Hessian quotient equations and using a P function. Our result generalizes the overdetermined problems for $k$-Hessian equations and $k$-curvature equations.
\end{abstract}
\maketitle
%\tableofcontents 

\section{Introduction}

In the seminal paper \cite{Serrin1971}, Serrin established the symmetry of the solution to 
\begin{equation}\label{laplace}
    \Delta u=n
\end{equation}
in a bounded $C^2$ domain $\Omega\subset\mathbb R^n$ with 
\begin{equation}\label{bdrycd}
u=0 \quad\text{and}\quad \frac{ \partial u}{\partial \gamma}=1\quad\text{on }\ \partial\Omega.
\end{equation} 
where $\gamma$ is the unit outer normal to $\partial\Omega$. 
 If $u\in C^2(\overline\Omega)$ is a solution to \eqref{laplace} and \eqref{bdrycd}, Serrin \cite{Serrin1971} proved that upto a translation $u=\frac{|x|^2-1}2$ and $\Omega$ is the unit ball centered at the origin. The proof is based on the method of \textit{moving planes} and it can be applied to more general uniformly elliptic equations. In \cite{Weinberger1971}, Weinberger provided an alternative proof by using maximum principle for P function and a Rellich-Pohozaev type identity.

There have been many generalizations of Serrin and Weinberger's work to  quasilinear elliptic equations (see e.g. \cite{FK2008,FGK2006,Garofalo1989} and reference therein ) and fully nonlinear equations  such as Hessian equation and Weingarten curvature equation (see e.g. \cite{Salani2008,JiaXiaohan2020,Bao2014}). 
In the Euclidean space, the overdetermined  problem for the $k$-Hessian equation i.e. $S_k(D^2u)=C_n^k$,  was studied in  \cite{Salani2008} by using a Rellich-Pohozaev type identity and Newton inequalities  and was also dealt in \cite{Bao2014} by using the method of moving planes. Using the $P$ function $P=|Du|^2-2u$ as mentioned in \cite{Ma1999, Weinberger1971}, we can prove the overdetermined problems to Hessian quotient equations in the  Euclidean space and in the hyperbolic space.

P functions have been extensively investigated and inspired several effective works in the context of elliptic partial differential equations. For fully nonlinear equations, the P function for 2-dimensional Monge-Amp\`ere equation was given by Ma \cite{Ma1999}, and the P functions for $k$-Hessian equations and $k$-curvature equations were given by Philippin and Safoui \cite{Safoui}. 

Let $\Omega$ be a bounded $C^2$ domain, $\gamma$ be the unit outer normal to $\partial\Omega$. Let $k,l$ be integers such that $0\leq l<k\leq n$.
In the first part of this paper, we consider the following overdetermined problem for Hessian quotient equations in the Euclidean space $\mathbb R^n$,
\begin{equation}\label{mainpb}
    \begin{aligned}
        \begin{cases}
            S_k(D^2u)=\frac{C_n^k}{C_n^l} {S_l(D^2u)}&\quad\text{in }\Omega,\\
            u=0&\quad\text{on }\partial \Omega,\\
           \frac{ \partial u}{\partial \gamma}=1&\quad\text{on }\partial\Omega.
        \end{cases}
    \end{aligned}
\end{equation}
%where $\Omega$ is a bounded $C^2$ domain of $\mathbb R^n$ and $k, l$ are integers  satisfying  $0\le l< k\le n$, $\gamma$ is the unit outer normal of $\partial\Omega$.

%The equation in Problem \eqref{mainpb} is called the Hessian quotient equation. If $l=0$ and $k=1$, it is the classical Laplacian equation. If  $l=0$ and $2\le k\le n$, it is called the $k$-Hessian equation. 

Our first result is the following.
\begin{theorem}\label{mainthm}
Let $\Omega$ be a $C^2$ bounded domain in $\mathbb R^n$, $u\in C^3(\Omega)\cap C^2(\overline\Omega)$ be a solution to \eqref{mainpb} with the integer $k, l$ satisfying $0\le l< k\le n$ and $S_{l}(D^2 u)>0$ in $\overline\Omega$. Then upto a translation $u=\frac{|x|^2-1}2$ and $\Omega$ is the unit ball with the center at $0$.
\end{theorem}
Since $S_{l}(D^2 u)>0$, similar as the argument in \cite{Salani2008}, we can prove that $u$ is $k$-convex which means $S_i(D^2 u)>0, 1\le i\le k$ in $\overline\Omega$. 
Then by maximum principle, $u<0$ in $\Omega$, and the solution to Dirichlet problem of $S_k(D^2u)=\frac{C_n^k}{C_n^l} {S_l(D^2u)}$ is unique. 

In the second part, we consider Hessian quotient equations in the hyperbolic space $\mathbb H^n$,
\begin{align}\label{mainpb2}
    \begin{cases}
    S_k(D^2u-uI)=\frac{C_n^k}{C_n^l}S_l(D^2u-uI)&\quad\text{in }\Omega,\\
    u=0&\quad\text{on }\partial\Omega,\\
    \frac{\partial u}{\partial \gamma}=1&\quad\text{on }\partial\Omega.
    \end{cases}
\end{align}
Our result is as follows.
\begin{theorem}\label{mainthm2}
Let $\Omega$ be a $C^2$ bounded domain in $\mathbb H^n$,  $u\in C^3(\Omega)\cap C^2(\overline\Omega)$ be a solution to \eqref{mainpb} with the integer $k, l$ satisfying $0\le l< k\le n$ and $S_{l}(D^2 u-uI)>0$ in $\overline\Omega$. Then upto a translation $u=\frac{\cosh r}{\cosh R}-1$ and $\Omega$ is the ball of radius $\tanh^{-1}1$ with the center at $0$.
\end{theorem}

In the third part, we consider Hessian quotient curvature equations in the Euclidean space $\mathbb R^n$,
\begin{align}\label{mainpb3}
    \begin{cases}
    S_k(D(\frac{Du}{w}))=\frac{C_n^k}{C_n^l}S_l(D(\frac{Du}{w}))&\quad\text{in }\Omega,\\
    u=0&\quad\text{on }\partial\Omega,\\
    \frac{\partial u}{\partial \gamma}=1&\quad\text{on }\partial\Omega.
    \end{cases}
\end{align}
Our result is as follows.
\begin{theorem}\label{mainthm3}
Let $\Omega$ be a $C^2$ bounded domain in $\mathbb R^n$,  $u\in C^3(\Omega)\cap C^2(\overline\Omega)$ be a solution to \eqref{mainpb} with the integer $k, l$ satisfying $0\le l< k\le n$ and $S_{l}(D(\frac{Du}{w}))>0$ in $\overline\Omega$. Then upto a translation $u=-\sqrt{1-|x|^2}+\frac1{\sqrt{2}}$ and $\Omega$ is the ball of radius $\frac1{\sqrt2}$.
\end{theorem}

In the  Section 2, we recall some notations concerning the Hessian operator in the Euclidean space and the hyperbolic space and some known facts about Weingartem hypersurfaces in the Euclidean space. Some known facts about curvature of level sets and Minkowskian integral formulas are also introduced there. In  Section \ref{sec3}, we first prove  a Rellich-Pohozaev type identity for  Problem \eqref{mainpb} in the Euclidean space. Then by the Rellich-Pohozaev type identity and the P function, we prove the Theorem \ref{mainthm} in the Euclidean space. In Section \ref{sec4}, we  derive the Rellich-Pohozaev type identity for Problem \eqref{mainpb2} in the hyperbolic space, then combing with a P function, we  give the proof of Theorem \ref{mainthm2}. In the last section, we prove Theorem \ref{mainthm3}.

\section{Preliminaries}\label{sec2}

\subsection{Elementary symmetric functions}
Let $\mathcal S^n$ be the space of real symmetric $n\times n$ matrices. We denote by $A=(a_{ij})$ a matrix in $\mathcal S^n$, and by $\lambda_1,\cdots, \lambda_n$ the eigenvalues of $A$. For $1\leq k\leq n$, we recall the definition of $k$-th elementary symmetric functions of $A$,

\begin{equation} S_k(A)=S_k(\lambda_1,\cdots,\lambda_n)=\sum_{1\leq i_1<\cdots<i_k\leq n}\lambda_{i_1}\cdots\lambda_{i_k}.\end{equation}

Denote by
\begin{equation}S_k^{ij}(A):=\frac{\partial S_k(A)}{\partial a_{ij}},\end{equation}
then it is easy to see from the definition above that
\begin{equation}\label{kl::1}
\begin{aligned}
\sum_{i,j=1}^n S_k^{ij}(A)a_{ij}=&kS_k(A),\quad\text{and}\quad 
\sum_{i=1}^n S_k^{ii}(A)=&(n-k+1)S_{k-1}(A).
\end{aligned}\end{equation}

For $1\leq k\leq n-1$, Newton's inequalities say that
\begin{equation}\label{Newton}(n-k+1)(k+1)S_{k-1}(A)S_{k+1}(A)\leq k(n-k)S_k^2(A).\end{equation}

For $1\leq k\leq n$, recall that the Garding cone is defined as
\begin{equation}\Gamma_k=\{A\in\mathcal S^n: S_1(A)>0,\cdots,  S_k(A)>0\}.\end{equation}
For $A\in\Gamma_k$ and $k>l\geq 0$, $r>s\geq 0$, $k\geq r$, $l\geq s$, we have the following Mclaurin inequalities
\begin{align}\label{gNMieq}
    \bigg(\frac{S_k(A)/C_n^k}{S_l(A)/C_n^l}\bigg)^\frac1{k-l}\leq \bigg(\frac{S_r(A)/C_n^r}{S_s(A)/C_n^s}\bigg)^\frac1{r-s}.
\end{align}
In particular, 
\begin{equation}\label{MacLaurin}
\bigg(\frac{S_k(A)}{C_n^k}\bigg)^{\frac{1}{k}}\leq \bigg(\frac{S_l(A)}{C_n^l}\bigg)^{\frac{1}{l}},\quad\forall\ k\geq l\geq 1,\end{equation}
and
\begin{equation}\label{NMieq}
    \frac{{S_k(A)}/{C_n^k}}{{S_{k-1}(A)}/{C_n^{k-1}}}\leq \frac{{S_l(A)}/{C_n^l}}{{S_{l-1}(A)}/{C_n^{l-1}}},\quad\forall \ k\geq l\geq 1.
\end{equation}
By \eqref{gNMieq}, we have 
\begin{align}\label{gNMieq1}
\frac{S_{k+1}(A)/C_n^{k+1}}{S_k(A)/C_n^k}\leq \bigg(\frac{S_k(A)/C_n^k}{S_l(A)/C_n^l}\bigg)^\frac1{k-l}.
\end{align}
The equalities in \eqref{gNMieq}, \eqref{MacLaurin}, \eqref{NMieq} and  \eqref{gNMieq1}   hold if the eigenvalues $\lambda_1,\cdots,\lambda_{n}$ of $A$ are equal to each other.
We need the following useful proposition which can be found in \cite{Reilly1974}.
\begin{proposition} \label{keyprop}
For any $n\times n$  matrix $A$, we have
\begin{equation}\label{prop2.1}
    S_k^{ij}(A)=S_{k-1}(A)\delta_{ij}-\sum_{l=1}^nS_{k-1}^{il}(A)a_{jl}.
\end{equation}
\end{proposition}

In the following we write $D, D^2$ and $\Delta$ for the gradient, Hessian and Laplacian on $\mathbb R^n$. 
% For simplicity, we will use $u_i, \ u_{ij},\ \cdots $ and $u_\gamma$ to denote covariant derivatives and normal derivative of function $u$ with respect to the metric on $\mathbb H^n$. We write $X\cdot Y$ instead of $g(X,Y)$ for vector fields $X,Y$. 
We also follow Einstein's summation convention.

\subsection{Hessian operators}
\subsubsection{Hessian operators in Euclidean space}
Let $\Omega$ be an open subset of $\mathbb R^n$ and let $u\in C^2(\Omega)$. The $k$-Hessian operator $ S_k(D^2u)$ is defined as the $k$-th elementary symmetric function of $D^2u$. Notice that 
\begin{equation} 
S_1(D^2u)=\Delta u\quad\text{and}\quad  S_n(D^2u)=\mathrm{det}D^2u.
\end{equation}

A function $u$ is called $k$-convex in $\Omega$, if $D^2u(x)\in \Gamma_k$ for any $x\in \Omega$. A direct computation yields that $(S_k^{1j}(D^2u),\cdots,S_k^{nj}(D^2u))$ is divergence free, that is 
\begin{equation}\label{f2.10}
    \frac{\partial}{\partial x_i}S_k^{ij}(D^2u)=0.
\end{equation}
By using \eqref{prop2.1}, it is easy to see that
\begin{equation}\label{f2.11}
    S_k^{ij}(D^2u)u_{il}=S_k^{il}(D^2u)u_{ij}.
\end{equation}
%\textcolor{blue}{For convenience, we use $S_k$ and $S_k^{ij}$ instead of $S_k(D^2u)$ and $S_k^{ij}(D^2u)$ respectively in the following.} 

\subsubsection{Hessian operators in hyperbolic space}
Let $\Omega$ be an open subset of $\mathbb H^n$ and let $u\in C^2(\Omega)$. The $k$-Hessian operator $S_k[u]$ is defined as the $k$-th elementary symmetric function $\sigma_k(D^2u-uI)$ of $D^2u-uI$. Notice that 
 \begin{equation}S_1[u]=\Delta u-nu\quad\text{and}\quad S_n[u]=\mathrm {det}(D^2u-uI).\end{equation}

 A function $u$ is called \textit{$k$-admissible} in $\Omega$ if $D^2u(x)-u(x)I\in\Gamma_k$ for any $x\in \Omega$. A function $u\in C^2(\Omega)$ is called an \textit{admissible solution} of $S_k[u]=f$ in $\Omega$, if $u$ solves the equation and $u$ is an $k$-admissible function.

We list the following propositions proven in \cite{GJY2022}, which will be used in  Section \ref{sec4}.

 \begin{proposition}\label{propdivfree}Suppose $u\in C^3(\Omega)$, then
 \begin{equation}\label{divfree}
     D_i(\sigma_k^{ij}(D^2u-uI))=0.
 \end{equation}
 \end{proposition}

% \begin{proof}
% We prove it by induction. For $k=1$, \eqref{divfree} holds obviously. For $k=2$, since $u_{ilj}=u_{ijl}-u_l\delta_{ij}+u_j\delta_{il}$, we have
%  \begin{equation}\begin{aligned}D_j\sigma_2^{ij}(D^2u-uI)=&D_j(\sigma_1(D^2u-uI)\delta_{ij}-\sigma_1^{js}(u_{si}-u\delta_{si}))\\=&(\Delta u-nu)_i-(u_{ji}-u\delta_{ji})_{j}=0.\end{aligned}\end{equation}
% Suppose that \eqref{divfree} holds for $k-1$, then
% \begin{equation}\begin{aligned}
%  D_j(\sigma_k^{ij}(D^2u-uI))=&D_j(\sigma_{k-1}(D^2u-uI)\delta_{ij}-\sigma_{k-1}^{js}(D^2u-uI)(u_{si}-u\delta_{si}))\\
%  =&D_i\sigma_{k-1}(D^2u-uI)-\sigma_{k-1}^{js}(D^2u-uI)(u_{si}-u\delta_{si})_{j}\\
%  =&\sigma_{k-1}^{js}(D^2u-uI)((u_{js}-u\delta_{js})_i-(u_{si}-u\delta_{si})_{j})\\
%  =&0.
%  \end{aligned}\end{equation}
% \end{proof}

 \begin{proposition}
 Let $u\in C^2(\Omega)$, then
 \begin{equation}\label{changeindex}
     \sigma_k^{ij}(D^2u-uI)u_{il}=\sigma_k^{il}(D^2u-uI)u_{ij}.
 \end{equation}
 \end{proposition}

% \begin{proof}
% By proposition \ref{keyprop}, we obtain
% \begin{equation}\begin{aligned}
% \sigma_k^{ij}(D^2u-uI)(u_{il}-\delta_{il})
% =&\big(\sigma_{k-1}(D^2u-uI)\delta_{ij}-\sigma_{k-1}^{is}(D^2u-uI)(u_{js}-u\delta_{js})\big)(u_{il}-\delta_{il})\\
% =&\big(\sigma_{k-1}(D^2u-uI)\delta_{sl}-\sigma_{k-1}^{is}(D^2u-uI)(u_{il}-u\delta_{il})\big)(u_{js}-\delta_{js})\\
% =&\sigma_{k}^{sl}(D^2u-uI)(u_{is}-u\delta_{js}).
% \end{aligned}\end{equation}
% Thus,
% \begin{equation}\begin{aligned}
%     \sigma_k^{ij}(D^2u-uI)u_{il}=&\sigma_k^{ij}(D^2u-uI)(u_{il}-u\delta_{il})+\sigma_k^{lj}(D^2u-uI)u\\=&\sigma_k^{il}(D^2u-uI)(u_{ij}-u\delta_{ij})+\sigma_k^{lj}(D^2u-uI)u\\=&\sigma_k^{il}(D^2u-uI)u_{ij}.
% \end{aligned}\end{equation}
% \end{proof}

\subsection{Weingarten hypersurfaces}
Let $\mathcal M$ be a hypersurface in $\mathbb R^{n+1}$, which is locally represented as a graph $x_{n+1}=u(x_1,\cdots, x_n$, $x=(x_1,\cdots,x_n)\in\Omega$, where $\Omega\subset \mathbb R^n$. Denote $u_i=\frac{\partial u}{\partial x_i}$, $u_{ij}=\frac{\partial^2 u}{\partial x_i\partial x_j}$, and $Du=(u_1,\cdots, u_n)$. Then $\nu=\frac{(-Du,1)}{\sqrt{1+|Du|^2}}$ is the outer unit normal of $\mathcal M$. Denote $w=\sqrt{1+|Du|^2}$. The furst and secondfundamental forms can be respectively expressed as
\begin{align}
    g_{ij}=\delta_{ij}+u_iu_j,\quad\text{and}\quad
    b_{ij}=\frac{u_{ij}}{w}.
\end{align}
Then the principal curvatures $\lambda _1,\cdots, \lambda_n$ of $\mathcal M$ are the eigenvalues of the second fundamental form relative to the first fundamental form, i.e. the eigenvalues of $g^{ik}b_{kj}=:a_{ij}$, where $g^{ij}=\delta_{ij}-\frac{u_i}{u_j}{w^2}$ is the inverse mateix to $g_{ij}$.
Hence
\begin{align}
    a_{ij}(u)=g^{ik}b_{kj}=\frac{u_{ij}}{w}-\frac{u_iu_ku_{kj}}{w^3}=\big(\frac{u_i}{w}\big)_j,
\end{align}
where $\big(\frac{u_i}{w}\big)_j=\frac{\partial }{\partial x_j}\big(\frac{u_i}{w}\big)$. We say $u$ is $k$-admissible if $a_{ij}(u)\in\Gamma_k$.

Reilly [Reilly1973] gave the following result.
\begin{proposition}
Suppose $u\in C^3(\Omega)$, $A_{ij}=(\frac{u_i}{w})_j$. Then
\begin{align}
    D_jS_k^{ij}(A)=0.
\end{align}
\end{proposition}

For general (non-symmeric) matrices, Pietra, Gavitone and Xia proved the following result in \cite{Xia2021}.
\begin{proposition}
For any $n\times n$ matrix $A=(a_{ij})$, we have
\begin{equation}
    S_k^{ij}(A)=S_{k-1}(A)\delta_{ij}-\sum_{l=1}^nS_{k-1}^{il}(A)a_{jl}.
\end{equation}
\end{proposition}
The following proposition can be inferred from above proposition, we omit the proof.
\begin{proposition}\label{prop::2.2}
For any $n\times n$ matrix $A=(a_{ij})$, we have
\begin{equation}\label{eq::14}
    S_k^{il}(A)a_{jl}=S_k^{lj}(A)a_{li}.    
\end{equation}
\end{proposition}

\subsection{Minkowskian integral formulas}
Let $\Omega$ be a $C^2$ bounded domain, and $\partial\Omega$ is the boundary of $\Omega$. Denote the principle curvatures of $\partial\Omega$ by $\kappa=(\kappa_1,\cdots,\kappa_{n-1})$. For $1\le k\le n-1$, the $k$-th curvature of $\partial\Omega$ is defined as
\begin{equation}H_k:=S_k(\kappa).\end{equation}
$\Omega$ is called \textit{$k$-convex}, if $H_i>0$ for all $1\le i\le k$. In particular, $(n-1)$-convex is strictly convex, 1-convex is also called mean convex.

In the theory of convex bodies and differential geometry, Minkowskian integral formula (see \cite{Hsiung1954,Hsiung1956,Schneider1993}) is very useful. Suppose $\Omega$ is a bounded $C^2$ domain of $\mathbb R^n$, then the Minkowskian integral formula says
\begin{equation}\label{Minkowskionrn}
    \int_{\partial\Omega}\frac{H_k}{C_{n-1}^{k}}x\cdot\gamma \mathrm d \sigma=\int_{\partial\Omega}\frac{H_{k-1}}{C_{n-1}^{k-1}} \mathrm d\sigma.
\end{equation}
Suppose $\Omega$ is a domain of $\mathbb H^n$. Let $p\in \Omega$, $r$ be the distance from $p$. Let $V(x)=\cosh(r(x))$. Then the Minkowskian integral formula says
 \begin{equation}\label{Minkowskionhn}
     \int_{\partial\Omega}\frac{H_k}{C_{n-1}^{k}}V_\gamma\mathrm d \sigma=\int_{\partial\Omega}\frac{H_{k-1}}{C_{n-1}^{k-1}}V \mathrm d\sigma.
 \end{equation}

\subsection{Curvatures of level sets}
Let $u$ be a smooth function in space form $\mathbb R^n$, for any regular $c\in\mathbb R$ of $u$ (that is, $Du(x)\neq 0$ for any $x\in\mathbb R^n$ such that $u(x)=c$),
 the level set $\Sigma_c:=u^{-1}(c)$ is a smooth hypersurface by the implicit function theorem. The $k$-th order curvature $H_k$ of the level set $\Sigma_c$ is given by
\begin{equation}\label{curvoflvlset}
    H_{k-1}=\frac{S_k^{ij}u_iu_j}{|Du|^{k+1}},
\end{equation}
which can be found in \cite{MZ2014,Reilly1974}.

At the last of this section, we introduce some notations. For convenience, we use $S_m$ and $S_m^{ij}$ instead of $S_m(D^2u)$ and $S_m^{ij}(D^u)$ in Section \ref{sec3},  instead of  instead of $S_m(D^2u-uI)$ and $S_m^{ij}(D^u-uI)$ in Section \ref{sec4},  instead of $S_m(D(\frac{Du}{w}))$ and $S_m^{ij}(D(\frac{DU}{w}))$ in Section \ref{sec5}.

\section{Overdetermined problem for Hessian quotient equations in Euclidean space}\label{sec3}

In this section, we present a Rellich-Pohozaev type identity for Hessian quotient equations with zero Dirichlet boundary condition, and use a $P$-function to give a proof of Theorem \ref{mainthm}. 
%For convenience, we use $S_k$ and $S_k^{ij}$ instead of $S_k(D^2u)$ and $S_k^{ij}(D^2u)$ respectively in the following.

The following lemma is proven in \cite{Salani2008}, which implies the solution to \eqref{mainpb} is $k$-convex. This ensures MacLaurin inequalities \eqref{MacLaurin} can be applied.
\begin{lemma}[\cite{Salani2008}]\label{ukconvex}
Let $\Omega\subset\mathbb R^n$ be a bounded $C^2$ domain and $u\in C^2(\overline\Omega)$ is a solution to \eqref{mainpb} with $S_l(D^2 u)>0$ in $\overline\Omega$, then $u$ is $k$-convex in $\Omega$.
\end{lemma}
\begin{remark}
The proof is amost the same as that in \cite{Salani2008}, except that a key inequality turns into
\begin{align*}
0<\frac1{S_l}(u_{nn}H_{k-1}+H_k).
\end{align*}
So we need $S_l>0$ to ensure that $u_{nn}\geq -\frac{H_k}{H_{k-1}}$.
\end{remark}

Based on Lemma \ref{ukconvex}, we are able to 
derive the lemma below, which was proved in \cite{Safoui}. So we can apply the maximum principle on the P function. For completeness, we present the proof. 

\begin{lemma}[\cite{Safoui}]\label{P}
Let $u\in C^3(\Omega)$ be an admissible (i.e. $u$ is $k$-convex ) solution of 
\begin{align}\label{kl::8}
    {S_k(D^2 u)}=\frac{C_n^k}{C_n^l}{S_l(D^2 u)}\quad   \text{in }\  \Omega\subset\mathbb R^n,
\end{align}
with $0\leq l<k\leq n$. Then the maximum of the following $P$-function 
\begin{equation}\label{eq:p}P:=|Du|^2-2u\end{equation} are obtained only on the boundary $\partial\Omega$ unless $P$ is constant.
\end{lemma}

\begin{proof}
Denote by 
\begin{align}
    F^{ij}:= \frac{\partial}{\partial u_{ij}}\frac{S_k(D^2 u)}{S_l(D^2 u)}=\frac1{S_l^2}(S_k^{ij}S_l(D^2 u)-S_k(D^2 u)S_l^{ij}).
\end{align}
By \eqref{kl::1}, we have
\begin{align}
 F^{ij}u_{ij}=& (k-l)\frac{S_k(D^2 u)}{S_l(D^2 u)},\\
  F^{ij}u_{si}u_{sj}=& \frac{1}{S_l^2}\big((l+1)S_{l+1}S_k-(k+1)S_{k+1}S_l\big).
\end{align}
Differenting the equation \eqref{kl::8}, we have 
\begin{align}
    F^{ij}u_{ijs}=0.\notag
\end{align}
Then  we have
\begin{align}\label{ineq1}
F^{ij}P_{ij}=&2F^{ij}\big(u_{si}u_{sj}+u_su_{sij}-u_{ij})\notag\\
=&2F^{ij} u_{si}u_{sj}-2F^{ij}u_{ij}\notag\\
=&\frac{2S_k(D^2 u)}{S_l(D^2 u)}\bigl((l+1)\frac{S_{l+1}(D^2 u)}{S_l(D^2 u)}-(k+1)\frac{S_{k+1}(D^2 u)}{S_k(D^2 u)}-(k-l)\bigr).
\end{align}
By \eqref{gNMieq} and \eqref{kl::8}, we have
\begin{align}
    \frac{S_{l+1}(D^2 u)/C_n^{l+1}}{S_l(D^2 u)/C_n^l}\geq \bigg(\frac{S_{k}(D^2 u)/C_n^k}{S_l(D^2 u)/C_n^l}\bigg)^\frac1{k-l}=1.
\end{align}
That is
\begin{align}\label{kl::5}
    (l+1)\frac{S_{l+1}(D^2 u)}{S_l(D^2 u)}\geq n-l.
\end{align}
Plugging  \eqref{kl::5}  into \eqref{ineq1}, we obtain
\begin{align*}\label{ineq2}
    F^{ij}P_{ij}\ge \frac{2S_k(D^2 u)}{S_l(D^2 u)}\bigl(n-k-(k+1)\frac{S_{k+1}(D^2 u)}{S_k(D^2 u)}\bigr)
\end{align*}
By \eqref{gNMieq1}, we have 
\begin{equation}\label{kl::6}
(k+1)\frac{S_{k+1}}{S_k}\leq n-k.
\end{equation}
Plugging  \eqref{kl::5} and \eqref{kl::6} into \eqref{ineq1}, we obtain
\begin{align}\label{kl::7}
    F^{ij}P_{ij}\geq 0.
\end{align}
By maximum principle, the maximum of $P$ is obtain on $\partial\Omega$. 
\end{proof}

The Rellich-Pohozaev type identity for $k$-Hessian equation has already been found \cite{Salani2008,Tso1990}. 
Brandolini, Nitsch, Salani and Trombetti \cite{Salani2008} gave the Rellich-Pohozaev type identity for $S_k(D^2u)=f(u)$ in $\Omega$, with $u=0$ on $\partial\Omega$,
\begin{equation}\frac{n-2k}{k(k+1)}\int_\Omega S_k^{ij}u_iu_j\mathrm dx+\frac1{k+1}\int_{\partial\Omega}x\cdot\gamma|Du|^{k+1}H_{k-1}\mathrm d\sigma=n\int_\Omega F(u)\mathrm dx,\end{equation} where $F(u)=\int_u^0f(s)\mathrm ds$.

For Hessian quotient equations, we first prove  identities for $\int_{\Omega}S_k(D^2 u)u$ and  $\int_{\Omega}S_l(D^2 u)u$ in which
the bad terms  $\int_{\Omega}\partial_s(S_k(D^2 u))x_s u$ and $\int_{\Omega}\partial_s(S_l(D^2 u))x_s u$ arise. By differenting the Hessian quotient equation, these two terms can be cancelled. Then we can prove the following  Rellich-Pohozaev type identity.

\begin{lemma}\label{dlempohorn}
Let $\Omega\subset \mathbb R^n $ be a bounded $C^2$ domain, $0\leq l<k\leq n$. If $u\in C^3(\Omega)\cap C^2(\overline\Omega)$ is a solution to the problem
\begin{equation}\label{f3.6}\begin{cases}
    {S_k(D^2 u)}=\frac{C_n^k}{C_n^l}{S_l(D^2 u)} &\quad\text{in }\Omega,\\
    u=0 &\quad\text{on }\partial\Omega.
\end{cases}\end{equation}
Suppose  $S_l(D^2 u)>0$ in $\overline\Omega$, then 
\begin{align}\label{PohoEuc}
    &(n-k+1)C_n^l\int_\Omega S_{k-1}(D^2 u) |Du|^2\mathrm dx-(n-l+1)C_n^k\int_\Omega S_{l-1}(D^2 u) |Du|^2\mathrm dx\nonumber\\
    &-C_n^l\int_{\partial\Omega}S_k^{ij} |Du|^2x_i\gamma_j\mathrm d\sigma+C_n^k\int_{\partial\Omega}S_l^{ij} |Du|^2x_i\gamma_j\mathrm d\sigma-2(k-l)C_n^l\int_\Omega S_k(D^2 u) u=0.
\end{align}
\end{lemma}

\begin{proof}
By direct computation, 
\begin{align}\label{f3.8}
    kS_ku=&S_k^{ij}u_{ij}u\nonumber\\=& S_k^{ij}u_{is}u(x_s)_{j}\nonumber\\
    =&(S_k^{ij}u_{is}ux_s)_j-S_k^{ij}u_{isj}ux_s-S_k^{ij}u_{is}u_j x_s\nonumber\\
    =&(S_k^{ij}u_{is}u x_s)_j-x_s\partial _sS_k(D^2 u)u-S_k^{ij}u_{is}u_jx_s.
\end{align}
By \eqref{f2.11}, we have
\begin{equation}\label{f3.11}
    S_k^{ij}u_{is}u_j x_s=S_k^{is}u_{ij}u_jx_s=\frac12S_k^{ij}(|Du|^2)_ix_j.
\end{equation}
Using \eqref{kl::1} and \eqref{f2.10}, we get
\begin{equation}\label{f3.9}
    S_k^{ij}(|Du|^2)_ix_j=(S_k^{ij}|Du|^2x_j)_i-(n-k+1)S_{k-1}|Du|^2.
\end{equation}
Putting \eqref{f3.8}, \eqref{f3.11} and \eqref{f3.9} together,  we find
\begin{align}\label{kl::2}
    2x_s\partial _sS_ku=&(S_k^{ij}u_{is}u(|x|^2)_s)_j -(S_k^{ij}|Du|^2x_j)_i+(n-k+1)S_{k-1}|Du|^2-2kS_ku .
\end{align}
Similarly, we obtain
\begin{align}\label{kl::3}
2x_s\partial _sS_lu=&(S_l^{ij}u_{is}u(|x|^2)_s)_j -(S_l^{ij}|Du|^2x_j)_i+(n-l+1)2S_{l-1}|Du|^2-2lS_lu .
\end{align}
Differentiating the equation \eqref{f3.6}, we have
\begin{align}\label{kl::4}
    C_n^l\partial_sS_k=C_n^k\partial_sS_l.
\end{align}
By \eqref{f3.6}, \eqref{kl::2}, \eqref{kl::3} and \eqref{kl::4}, we obtain
\begin{align}
    &(n-k+1)C_n^l\int_\Omega S_{k-1} |Du|^2\mathrm dx-(n-l+1)C_n^k\int_\Omega S_{l-1} |Du|^2\mathrm dx\nonumber\\
    &-C_n^l\int_{\partial\Omega}S_k^{ij} |Du|^2x_i\gamma_j\mathrm d\sigma+C_n^k\int_{\partial\Omega}S_l^{ij} |Du|^2x_i\gamma_j\mathrm d\sigma-2(k-l)C_n^l\int_\Omega S_k u=0
\end{align}
\end{proof}
The following lemma help us to deal with the term of boundary integral in \eqref{PohoEuc}.
\begin{lemma}\label{dlemma3.3}
Let $\Omega\subset \mathbb R^n$ be a $C^2$ bounded domain, $u\in C^2(\overline\Omega)$ be a solution to the equation \eqref{mainpb}, then 
\begin{equation}\label{formulaforlhs}
   S_k^{ij} x_j\gamma_i|Du|^2 =S_k^{ij} u_iu_jx\cdot\gamma\quad\text{on }\partial\Omega.
\end{equation}
\end{lemma}
\begin{proof}
By the boundary conditions of problem \eqref{mainpb}, we have on $\partial\Omega$ that
\begin{equation}\label{Duonpo}
    Du=u_\gamma\gamma,
\end{equation}
and 
\begin{equation}
    \begin{aligned}
    (u_\gamma)_j=(u_i\gamma_i)_j=u_{ij}\gamma_i+u_iD_j\gamma_i=u_{ij}\gamma_i+u_\gamma \gamma_iD_j\gamma_i=u_{ij}\gamma_i.
    \end{aligned}
\end{equation}
Then 
\begin{equation}
    \begin{aligned}
    (u_\gamma)_\gamma=u_{ij}\gamma_i\gamma_j:=u_{\gamma\gamma}.
    \end{aligned}
\end{equation}
Since $u_\gamma=c_0$ on $\partial\Omega$, for any tangential direction $\tau$, we have $(u_\gamma)_\tau=0$, then
\begin{equation}
    \begin{aligned}
    Du_\gamma=(u_\gamma)_\gamma\gamma=u_{ij}\gamma_i\gamma_j\gamma.
    \end{aligned}
\end{equation}
Therefore, 
\begin{equation}\label{3.10}
    \begin{aligned}
    u_{ij}\gamma_i=(u_\gamma)_j=u_{kl}\gamma_k\gamma_l\gamma_j=(u_{\gamma})_{\gamma}\gamma_j=u_{\gamma\gamma}\gamma_j.
    \end{aligned}
\end{equation}
It follows that 
\begin{equation}\label{3.11}u_{ij}u_i=u_{\gamma\gamma}u_j.\end{equation}
Now we derive \eqref{formulaforlhs} by induction. When $k=1$, \eqref{formulaforlhs} holds obviously since $S_1^{ij} =\delta_{ij}$. Suppose \eqref{formulaforlhs} holds for $k-1$, then 
\begin{equation}\label{e3.15}\begin{aligned}
S_k^{ij} x_i\gamma_j|Du|^2=&(S_{k-1} \delta_{ij}-S_{k-1}^{js} u_{si})x_j\gamma_i|Du|^2\\
=&S_{k-1} |Du|^2x\cdot\gamma-S_{k-1}^{js} u_{\gamma\gamma}\gamma_sx_j|Du|^2\\
=&S_{k-1} |Du|^2x\cdot\gamma-S_{k-1}^{ij} u_iu_jx\cdot\gamma u_{\gamma\gamma}\\
=&(S_{k-1} |Du|^2-S_{k-1}^{ij} u_{il}u_lu_j)x\cdot\gamma\\
=&S_{k}^{ij} u_iu_jx\cdot\gamma,
\end{aligned}\end{equation}
where we use \eqref{prop2.1} in the first and last equality, \eqref{3.10} in the second equality, and \eqref{3.11} in the fourth equality.
\end{proof}

Moreover, by formulas for the curvature of level sets and Minkowskian integral formulas, the integral on the boundary can be converted  into the integral on $\Omega$.
\begin{lemma}\label{dlemma3.4}
Let $u\in C^3(\Omega)\cap C^2(\overline\Omega)$ be a solution to the problem \eqref{mainpb}, then 
\begin{equation}\label{d3.13}
    \int_{\partial\Omega}S_k^{ij} u_iu_jx\cdot\gamma\mathrm d\sigma=(n-k+1)\int_{\Omega}S_{k-1}(D^2 u) \mathrm dx.
\end{equation}
\end{lemma}
\begin{proof}
From \eqref{curvoflvlset}, 
\begin{equation}
    S_k^{ij}u_iu_jx\cdot \gamma=H_{k-1}|Du|^{k+1}x\cdot \gamma\quad\text{on }\partial\Omega,
\end{equation}
Integrating it on $\partial\Omega$, and using Minkowskian integral formula \eqref{Minkowskionrn}, we obtain
\begin{equation}\label{d3.15}
    \int_{\partial\Omega}S_k^{ij} u_iu_jx\cdot \gamma\mathrm d\sigma=\int_{\partial\Omega}H_{k-1}|Du|^{k+1}x\cdot \gamma\mathrm d\sigma =\frac{n-k+1}{k-1}\int_{\partial\Omega}H_{k-2}|Du|^{k+1}\mathrm d\sigma.
\end{equation}
By \eqref{curvoflvlset} and \eqref{Duonpo}, we have
\begin{equation}
    H_{k-2}|Du|^{k+1}=S_{k-1}^{ij} u_iu_j|Du|= S_{k-1}^{ij} u_i\gamma_j\quad\text{on }\partial\Omega.
\end{equation}
Applying the divergence theorem, we get
\begin{equation}\label{d3.17}
\begin{aligned}
\int_{\partial\Omega}H_{k-2}|Du|^{k+1}\mathrm  d\sigma=& \int_{\Omega}(S_{k-1}^{ij} u_i)_j\mathrm dx
\end{aligned}
\end{equation}
Since $S_k^{ij}$ is divergence free, it follows that
\begin{equation}\label{d3.18}
    (S_{k-1}^{ij} u_i)_j=S_{k-1}^{ij} u_{ij}=(k-1)S_{k-1} .
\end{equation}
Putting \eqref{d3.15}, \eqref{d3.17} and \eqref{d3.18} together, we deduce \eqref{d3.13}.
\end{proof}

\begin{proof}[Proof of Theorem \ref{mainthm}.]
By Lemma \ref{dlempohorn}, Lemma \ref{dlemma3.3} and Lemma \ref{dlemma3.4}, we obtain
\begin{equation}
\begin{aligned}
2(k-l)C_n^l\int_\Omega uS_k\mathrm dx=&
(n-k+1)C_n^l\int_\Omega S_{k-1}(|Du|^2-1)\mathrm dx\\
&-(n-l+1)C_n^k\int_\Omega S_{l-1}(|Du|^2-1)\mathrm dx,
\end{aligned}
\end{equation}
or equivalently,
\begin{align}\label{test2}
    2(k-l)\int_\Omega (-u)S_k(D^2 u)\mathrm dx=
    (n-k+1)\int_\Omega (1-|Du|^2)S_{k-1}(D^2 u)\Big(1-\frac{l}{k}\frac{S_{l-1}(D^2 u)/C_{n}^{l-1}}{S_{k-1}(D^2 u)/C_n^{k-1}}\Big)\mathrm dx.
\end{align}
By \eqref{mainpb} and \eqref{gNMieq}, we have
\begin{align}
    \bigg(\frac{S_{k-1}/C_n^{k-1}}{S_{l-1}/C_n^{l-1}}\bigg)^\frac1{k-l}\geq \bigg(\frac{S_{k}/C_n^k}{S_{l}/C_n^l}\bigg)^\frac1{k-l}=1\geq \frac{S_k/C_n^k}{S_{k-1}/C_n^{k-1}}.
\end{align}
So 
\begin{align}\label{kl::9}
    \frac{S_{l-1}}{S_{k-1}}\frac{(n-l+1)C_n^k}{(n-k+1)C_n^l}\leq \frac lk\quad\text{and}\quad S_k\leq \frac{n-k+1}kS_{k-1}.
\end{align}
By Lemma \ref{P}, 
\begin{align}
    1-|Du|^2\ge -2u> 0 \ \text{in }\ \Omega.
\end{align}
Then substituting \eqref{kl::9} into \eqref{test2}, we get
\begin{align}
    \int_\Omega S_{k-1}(D^2 u)(|Du|^2-2u-1)\mathrm dx\geq 0.
\end{align}
By Lemma \ref{P},
\begin{align}
    P\leq \max_{\partial\Omega}P=1.
\end{align}
It follows that
\begin{align}
    P=|Du|^2-2u\equiv 1\quad\text{in }\Omega.
\end{align}
Since the derivatives vanish, by \eqref{ineq1} and \eqref{kl::5}, we obtain
\begin{align}
    0\geq n-l-(k+1)\frac{S_{k+1}}{S_k}.
\end{align}
Hence $S_{k+1}>0$ and the equality in \eqref{kl::6} holds. By \eqref{gNMieq}, the eigenvalues of $D^2u$ are equal to $1$. Using the boundary condition in \eqref{mainpb}, we derive that $u=\frac{|x|^2-1}2$. Hence we complete the proof of Theorem \ref{mainthm}
\end{proof}

\section{Overdetermined problem for Hessian quotient equations on the Hyperbolic space}\label{sec4}

In this section, we present a Rellich-Pohozaev type identity for Hessian quotient equations with zero Dirichlet boundary condition, and use a $P$-function to give a proof of Theorem \ref{mainthm2}. 
%For convenience, we use $S_k$ and $S_k^{ij}$ instead of $S_k(D^2u)$ and $S_k^{ij}(D^2u)$ respectively in the following.

The following lemma can be proved amost the same as  \cite{GJY2022}, which implies the solution to \eqref{mainpb2} is $k$-admissible. 
\begin{lemma}
Let $\Omega\subset\mathbb H^n$ be a bounded $C^2$ domain and $u\in C^2(\overline\Omega)$ is a solution to the problem \eqref{mainpb}, then $u$ is $k$-admissibe in $\Omega$.
\end{lemma}

The following two  lemmas are from  \cite{GJY2022}.
\begin{lemma}\label{lemma::kl::7}
Let $u\in C^2(\overline\Omega)$  satisfying $u=0$ and $u_\gamma=1$ on $\partial\Omega$, then 
\begin{equation}
   S_k^{ij} V_j\gamma_i|Du|^2 =S_k^{ij} u_iu_jV_\gamma\quad\text{on }\partial\Omega.
\end{equation}
\end{lemma}

\begin{lemma}\label{lemma::kl::8}
Let $u\in C^2(\overline\Omega)$  satisfying $u=0$ and $u_\gamma=1$ on $\partial\Omega$, then 
\begin{equation}\label{d4.13}
    \int_{\partial\Omega}S_{k}^{ij} u_iu_jV_\gamma\mathrm d\sigma=(n-k+1)\int_\Omega S_{k-1} V\mathrm dx.
\end{equation}
\end{lemma}
$P=|Du|^2-u^2-2u$ was proven to be a P function for the $k$-Hessian equation on $\mathbb H^n$ in \cite{GJY2022}. In the following lemma, we prove it is also a P function for the Hessian quotient equation on $\mathbb H^n$. By the appointment in Section 2, we use $S_k$ and $S_k^{ij}$ instead of $S_k(D^2u-uI)$ and $S_k^{ij}(D^2u-uI)$ in this section.
\begin{lemma}\label{lemma::kl::9}
Let $u\in C^3(\Omega)$ be a solution to
\begin{align}\label{klh::1}
    S_k  =\frac{C_n^k}{C_n^l}S_l&\quad\text{in }\Omega,
\end{align}
with $0\leq l <k\leq n$. Then the maximum of  $|Du|^2-u^2-2u$ is obtained on $\partial\Omega$.
\end{lemma}
\begin{proof}
Let $F^{ij}=\frac{\partial }{\partial u_{ij}}\frac{S_k(D^2u-uI)}{S_l(D^2u-uI)}$, then
\begin{align}
    F^{ij}=\frac1{S_l^2}(S_k^{ij}S_l-S_kS_l^{ij}).
\end{align}
Let $A_{ij}=u_{ij}-u\delta_{ij}$ for short. Then by direct calculations,  we have
\begin{align}
    \frac12(|Du|^2-u^2-2u)_{ij}=A_{si}A_{sj}-A_{ij}+u_lA_{ij,l}+u(A_{ij}-\delta_{ij}).
\end{align}
Contracting with $F^{ij}$, we obtain
\begin{align}
    \frac12F^{ij}(|Du|^2-u^2-2u)_{ij}=&F^{ij}(A_{si}A_{sj}-A_{ij}+u_lA_{ij,l}+u(A_{ij}-\delta_{ij}))\nonumber\\
    =&F^{ij}(A_{si}A_{sj}-A_{ij})+(-u)F^{ij}(\delta_{ij}-A_{ij})\nonumber\\
    =&\frac1{S_l^2}\Big((S_1S_k-(k+1)S_{k+1})S_l-(S_1S_l-(l+1)S_{l+1})S_k-(k-l)S_kS_l\Big)\nonumber\\&-\frac{u}{S_l^2}\Big((n-k+1)S_{k-1}S_l-(n-l+1)S_{l-1}S_k-(k-l)S_kS_l\Big)
\end{align}
By \eqref{gNMieq} and \eqref{klh::1}, we have
\begin{align}
    \bigg(\frac{S_{k-1}/C_n^{k-1}}{S_l/C_n^l}\bigg)^\frac1{k-l-1}\geq \bigg(\frac{S_{k}/C_n^{k}}{S_l/C_n^l}\bigg)^\frac1{k-l}=1,
\end{align}
and
\begin{align}
    \bigg(\frac{S_{k}/C_n^{k}}{S_{l-1}/C_n^{l-1}}\bigg)^\frac1{k-l-1}\geq \bigg(\frac{S_{k}/C_n^{k}}{S_l/C_n^l}\bigg)^\frac1{k-l}=1.
\end{align}
So
\begin{align}
    \frac{S_{k-1}}{S_l}\geq \frac{C_n^{k-1}}{C_n^l},\quad\text{and }\quad \frac{S_{l-1}}{S_k}\leq \frac{C_n^{l-1}}{C_n^k}.
\end{align}
Hence
\begin{align}
    &-\frac{u}{S_l^2}\Big((n-k+1)S_{k-1}S_l-(n-l+1)S_{l-1}S_k-(k-l)S_kS_l\Big)\nonumber\\
%    =&(-u)\Big((n-k+1)\frac{S_{k-1}}{S_l}-(n-l+1)\frac{S_{l-1}S_k}{S_l^2}-(k-l)\frac{S_k}{S_l}\Big)\nonumber\\
    =&(-u)\Big((n-k+1)\frac{S_{k-1}}{S_l}-(n-l+1)\frac{S_{l-1}}{S_l}\big(\frac{C_n^k}{C_n^l}\big)^2-(k-l)\frac{C_n^k}{C_n^l}\Big)\nonumber\\
    \geq&(-u)\frac{C_n^k}{C_n^l}((n-k+1)\frac{C_n^{k-1}}{C_n^k}-(n-l+1)\frac{C_n^{l-1}}{C_n^l}-(k-l))\nonumber\\
    =&0.
\end{align}
So we have
\begin{align}
    F^{ij}P_{ij}\geq 0.
\end{align}
By maximum principle, the maximum of $P$ is obtained on $\partial\Omega$.
\end{proof}
We will use the following Rellich-Pohozaev type identity.
\begin{lemma}\label{lemma::kl::10}
Let $u\in C^1(\overline\Omega)\cap C^3(\Omega)$ be a solution of 
\begin{equation}\label{klh::2}
    \begin{cases}
    S_k  =\frac{C_n^k}{C_n^l}S_l&\quad\text{in }\Omega,\\
    u=0&\quad\text{on }\partial\Omega
    \end{cases}
\end{equation}
Then there holds
\begin{align}\label{Poho::klh}
    &\frac{n-k+1}2C_n^l\int_\Omega S_{k-1}(|Du|^2-u^2)V\mathrm dx-\frac{n-l+1}2C_n^k\int_\Omega S_{l-1}(|Du|^2-u^2)V\mathrm dx\nonumber\\
    -&C_n^l\int_{\partial\Omega}S_k^{is}|Du|^2V_s\gamma_i\mathrm d\sigma-C_n^k\int_{\partial\Omega}S_l^{is}|Du|^2V_s\gamma_i\mathrm d\sigma-(k-l)C_n^l\int_\Omega S_kuV\mathrm dx=0.
\end{align}
\end{lemma}
\begin{proof}
Multiplying the equation by $uV$, we obtain
\begin{align}
    kS_kuV=&S_k^{ij}A_{ij}uV\nonumber\\
    =&S_k^{ij}u_{ij}uV-(n-k+1)S_{k-1}u^2V.
\end{align}
Since $D^2V=VI$, we have
\begin{align}
    S_k^{ij}u_{ij}uV=&S_k^{ij}u_{il}uV_{lj}\nonumber\\
    =&(S_k^{ij}u_{il}uV_l)_j-S_k^{ij}u_{ilj}uV_l-S_k^{ij}u_{il}u_jV_l.
\end{align}
Using
$
    u_{ilj}=u_{ijl}-u_l\delta_{ij}+u_j\delta_{il}=(u_{ij}-u\delta_{ij})_l+u_j\delta_{il}
$,
we have
\begin{align}
    S_k^{ij}u_{ij}uV=&S_k^{ij}u_{il}uV_{lj}\nonumber\\
    =&(S_k^{ij}u_{il}uV_l)_j-S_k^{ij}(u_{ij}-u\delta_{ij})_luV_l-S_k^{ij}u_juV_i-S_k^{ij}u_{il}u_jV_l\nonumber\\
    =&(S_k^{ij}u_{il}uV_l)_j-uV_lD_lS_k-S_k^{ij}u_juV_i-S_k^{ij}u_{il}u_jV_l,
\end{align}
where we use the identity from differenting the equation $S_k=C_n^k$ in the second equality. Furhermore, we have
\begin{align}
    S_k^{ij}u_juV_i=&\frac12S_k^{ij}(u^2)_jV_i\nonumber\\
    =&\frac12(S_k^{ij}u^2V_i)_j-\frac12S_k^{ij}u^2V_{ij}\nonumber\\
    =&\frac12(S_k^{ij}u^2V_i)_j-\frac{n-k+1}2S_{k-1}u^2V,
\end{align}
and
\begin{align}
    S_k^{ij}u_{il}u_jV_l=&S_k^{il}u_{ij}u_jV_l\nonumber\\
    =&\frac12S_k^{ij}(|Du|^2)_iV_j\nonumber\\
    =&\frac12(S_k^{ij}|Du|^2V_j)_i-\frac12S_k^{ij}|Du|^2V_{ij}\nonumber\\
    =&\frac12(S_k^{ij}|Du|^2V_j)_i-\frac{n-k+1}2S_{k-1}|Du|^2V.
\end{align}
Substituting above together, we obtain
\begin{align}
    uV_sD_sS_k=(S_k^{ij}u_{is}uV_s)_j-\frac12(S_k^{ij}u^2V_i)_j-\frac12(S_k^{ij}|Du|^2V_i)_j+\frac{n-k+1}2S_{k-1}(|Du|^2-u^2)V-kS_kuV.
\end{align}
By \eqref{klh::2}, we have
\begin{align}
    C_n^lD_sS_k=C_n^kD_sS_l.
\end{align}
So
\begin{align}
    C_n^l\Big((S_k^{ij}u_{is}uV_s)_j-\frac12(S_k^{ij}u^2V_i)_j-\frac12(S_k^{ij}|Du|^2V_i)_j+\frac{n-k+1}2S_{k-1}(|Du|^2-u^2)V-kS_kuV\Big)\nonumber\\
    =C_n^k\Big((S_l^{ij}u_{is}uV_s)_j-\frac12(S_l^{ij}u^2V_i)_j-\frac12(S_l^{ij}|Du|^2V_i)_j+\frac{n-l+1}2S_{l-1}(|Du|^2-u^2)V-lS_luV\Big).
\end{align}
Integrating it on $\Omega$ and use \eqref{klh::2}, we final obtain \eqref{Poho::klh}, thus finish the proof.

\end{proof}

\begin{proof}[Proof of Theorem \ref{mainpb2}. ]
By Lemma \ref{lemma::kl::7}, Lemma \ref{lemma::kl::8}, and Lemma \ref{lemma::kl::10}, we obtain
\begin{align}
    &(n-k+1)C_n^l\int_\Omega S_{k-1}(|Du|^2-u^2-1)V\mathrm dx-(n-l+1)C_n^k\int_\Omega S_{l-1}(|Du|^2-u^2-1)V\mathrm dx\nonumber\\
    =&2(k-l)C_n^l\int_\Omega S_kuV\mathrm dx.
\end{align}
That is
\begin{align}
    (n-k+1)\int_\Omega S_{k-1}(|Du|^2-u^2-1)\Big(1-\frac{(n-l+1)C_n^kS_{l-1}}{(n-k+1)C_n^lS_{k-1}}\Big)V\mathrm dx=2(k-l)\int_\Omega S_kuV\mathrm dx
\end{align}
By  \eqref{gNMieq} and \eqref{klh::2}, we have
\begin{align}
    \bigg(\frac{S_{k-1}/C_n^{k-1}}{S_{l-1}/C_n^{l-1}}\bigg)^\frac1{k-l}\geq \bigg(\frac{S_{k}/C_n^k}{S_{l}/C_n^l}\bigg)^\frac1{k-l}=1\geq \frac{S_k/C_n^k}{S_{k-1}/C_n^{k-1}}.
\end{align}
So 
\begin{align}
    \frac{S_{l-1}}{S_{k-1}}\frac{(n-l+1)C_n^k}{(n-k+1)C_n^l}\leq \frac lk\quad\text{and}\quad S_k\leq \frac{n-k+1}kS_{k-1}.
\end{align}
Note that by Lemma \ref{lemma::kl::9}, we have 
\begin{align}
    |Du|^2-u^2-1\leq 2u\quad\text{in }\Omega, 
\end{align}
So
\begin{align}
    &(n-k+1)\int_\Omega S_{k-1}(|Du|^2-u^2-1)\Big(1-\frac{(n-l+1)C_n^kS_{l-1}}{(n-k+1)C_n^lS_{k-1}}\Big)V\mathrm dx\nonumber\\
    \leq&\frac{2(n-k+1)(k-l)}k\int_\Omega S_{k-1}uV\mathrm dx
\end{align}
On the other hand,
\begin{align}
    2(k-l)\int_\Omega S_kuV\mathrm dx\geq \frac{2(n-k+1)(k-l)}k\int_\Omega S_{k-1}uV\mathrm dx.
\end{align}
So we have
\begin{align}
    \frac{n-k+1}kS_{k-1}=S_k.
\end{align}
By \eqref{NMieq}, we infer the eigenvalues of $D^2u-uI$ are all equal to $1$. Follows from an Obata type result (\cite{Reilly1980}, See also \cite{Catino2012,CheegerColding1996,Ciraolo2019}), $\Omega$ must be a ball $B_R$ and $u$ depends only on the distance from the center of $B_R$, where $R=\tanh^{-1}1$. It is easy to see that $u$ is of the form 
\begin{equation}u=\frac{\cosh r}{\cosh R}-1.\end{equation}
\end{proof}

\section{Quotient curvature equations}

In this section, we present a Rellich-Pohozaev type identity for Hessian quotient equations with zero Dirichlet boundary condition, and use a $P$-function to give a proof of Theorem \ref{mainthm3}. 
%For convenience, we use $S_k$ and $S_k^{ij}$ instead of $S_k(D^2u)$ and $S_k^{ij}(D^2u)$ respectively in the following.

The following lemma can be proved amost the same as  \cite{JiaXiaohan2020}, which implies the solution to \eqref{mainpb3} is $k$-admissible.
\begin{lemma}
Let $\Omega$ be a $C^2$ bounded domain of $\mathbb R^n$ and $u\in C^2(\overline \Omega)$ be a solution to problem \eqref{mainpb3}, then $u$ is a $k$-admissible function in $\Omega$ and $\Omega$ is $(k-1)$-convex.
\end{lemma}
The following lemma is from \cite{JiaXiaohan2020}.
\begin{lemma}\label{lemma::curkl::12}
Let $\Omega\subset\mathbb R^n$ be a $C^2$ bounded domain, $u\in C^2(\overline\Omega)$ satisfies $u=0$ and $u_\gamma=1$ on $\partial\Omega$, then
\begin{align}
    \int_{\partial\Omega}\frac{S_k^{ij}x_j\gamma_i}{w}\mathrm d\sigma=\frac{n-k+1}{\sqrt 2}\int_\Omega S_{k-1}\mathrm dx.
\end{align}
\end{lemma}
In \cite{JiaXiaohan2020}, Jia prove that $P=\frac1w+u$ is a P function to constant curvature equation $S_k(D(\frac{Du}{w}))=C_n^k$. The following lemma implies it is also a P function for $S_k(D(\frac{Du}{w}))=\frac{C_n^k}{C_n^l} S_l(D(\frac{Du}{w}))$.  As agreed in Section 2, we use $S_k$ and $S_k^{ij}$ instead of $S_k(D(\frac{Du}{w}))$ and $S_k^{ij}(D(\frac{Du}{w}))$ in this section.
\begin{lemma}\label{lemma::curkl::13}
Let $\Omega$ be a $C^2$ bounded domain of $\mathbb R^n$, $u\in C^2(\overline \Omega)\cap C^3(\Omega)$ be a solution to problem \eqref{mainpb3}, the minima of $\frac1w$, $-u$  and $p:=\frac12+u$ is obtained on $\partial\Omega$.
\end{lemma}
\begin{proof}
Let $\vec X=(x,u)\in\mathcal M$ and $e_{n+1}=(0,\cdots, 0,1)\in\mathbb R^n$. Choose a local orthonormal frame $\{e_1,\cdots.e_n\}$  on $\mathcal M$. The  unit normal is $\vec N=\frac{(-Du,1)}{w}$. 
Then we have
\begin{align}
    \frac1w=<\vec B,e_{n+1}>, \quad u=\langle \vec X,e_{n+1}\rangle ,\quad \vec X_i=e_i,
\end{align}
and
\begin{align}
    e_{ij}=h_{ij}\vec N,\quad \vec N_i=-h_{ij}e_j,
\end{align}
where $h_{ij} $ is the coefficient of the second fundamental form. Thus
\begin{align}
    P=\frac1w+u=\langle \vec N,e_{n+1}\rangle +\langle \vec X,e_{n+1}\rangle ,
\end{align}
%\begin{align}
%    p_i=-h_{il}\langle  e_l,e_{n+1}\rangle +\langle  e_i,e_{n+1}\rangle ,
%\end{align}
By direct computation, we obtain
\begin{align}
    P_{ij}=&-h_{il,j}\langle e_l,e_{n+1}\rangle -h_{il}\langle  e_{lj},e_{n+1}\rangle+\langle e_{ij},e_{n+1}\rangle \nonumber\\
    =&-h_{ij,l}\langle  e_l,e_{n+1}\rangle -h_{is}h_{sj} \langle  \vec N,e_{n+1}\rangle +h_{ij}\langle  \vec N,e_{n+1}\rangle \nonumber\\
    =&-h_{ij,l}\langle e_l,e_{n+1}\rangle -\frac1wh_{is}h_{sj}+\frac1wh_{ij},
\end{align}
where we used the Codazzi equation $h_{ij,l}=h_{il,j}$ in the last inequality.
Let $F^{ij}=\frac{\partial}{\partial h_{ij}}\frac{S_k(h_{st})}{S_l(h_{rq})} $, then
\begin{align}
    F^{ij}=\frac1{S_l^2}(S_k^{ij}S_l-S_l^{ij}S_k)
\end{align}
Therefore
\begin{align}
    F^{ij}P_{ij}=&F^{ij}(-h_{ij,l}\langle e_l,e_{n+1}\rangle -\frac1wh_{is}h_{sj}+\frac1wh_{ij})\nonumber\\
    =&F^{ij}(-\frac1wh_{is}h_{sj}+\frac1wh_{ij})\nonumber\\
    =&\frac1{wS_l^2}\big((S_k^{ij}S_l-S_kS_l^{ij})(h_{ij}-h_{is}h_{sj})\big)\nonumber\\
    =&\frac1{wS_l^2}\big((k-l)S_kS_l+(l+1)S_{l+1}S_k-(k+1)S_{k+1}S_l\nonumber\\
    \leq &0
\end{align}
By maximum principle, the minimum of $p$ is obtain on $\partial\Omega$. The same argeument leads the conclusion for $\frac1w $ and $-u$.
\end{proof}
We prove the following Pohozaev-type identity.
\begin{lemma}\label{lemma::curkl::14}
Let $\Omega\subset\mathbb R^n$ be a $C^2$ domain, $u\in C^2(\overline\Omega)\cap C^3(\Omega)$ be a solution to the following problem
\begin{equation}\label{mainpb::curkl}
    \begin{cases}
    S_k=\frac{C_n^k}{C_n^l}S_l\quad\text{in }\Omega,\\
    u=0\quad\text{on }\partial\Omega,
    \end{cases}
\end{equation}
with $0\leq l<k\leq n$. Then
\begin{align}\label{curkl::1}
    &C_n^l\int_{\partial\Omega}S_k^{si}\frac{x_s\gamma_i}w\mathrm d\sigma-C_n^k\int_{\partial\Omega}S_l^{si}\frac{x_s\gamma_i}w\mathrm d\sigma-(k-l)C_n^l\int_\Omega uS_k\mathrm dx\nonumber\\
    &-(n-k+1)C_n^l\int_\Omega \frac{S_{k-1}}w+(n-l+1)C_n^k\int_\Omega \frac{S_{l-1}}w=0.
\end{align}
\end{lemma}
\begin{proof}
By direct computation, we obtain
\begin{align}
    (\frac{u_i}{w})_ju_i=\frac12w\big(\frac{|Du|^2}{w^2}\big)_j=\frac12w\big(1-\frac1{w^2}\big)_j=-(w^{-1})_i.
\end{align}
Multiplying the equation with $ku$, we obtain
\begin{align}
    kS_ku=&S_k^{ij}\big(\frac{u_i}w\big)_ju\nonumber\\
    =&S_k^{ij}\big(\frac{u_i}w\big)_s(\frac{|x|^2}2)_{sj}u\nonumber\\
    =&(S_k^{ij}\big(\frac{u_i}w\big)_sx_su)_j-S_k^{ij}\big(\frac{u_i}w\big)_{sj}x_su-S_k^{ij}\big(\frac{u_i}w\big)_sx_su_j\nonumber\\
    =&(S_k^{ij}\big(\frac{u_i}w\big)_sx_su)_j-ux_sD_sS_k-S_k^{ij}\big(\frac{u_i}w\big)_sx_su_j,\nonumber\\
\end{align}
where
\begin{align}
    S_k^{ij}\big(\frac{u_i}w\big)_sx_su_j=&S_k^{si}\big(\frac{u_j}w\big)_ix_su_j\nonumber\\
    =&-S_k^{si}(w^{-1})_ix_s\nonumber\\
    =&-(S_k^{si}(w^{-1})x_s)_i+\frac{(n-k+1)S_{k-1}}w\nonumber\\
\end{align}
Putting above identity together, we obatin
\begin{align}
ux_sD_sS_k=(S_k^{ij}\big(\frac{u_i}w\big)_sx_su)_j+(S_k^{si}\frac{x_s}{w})_i-(n-k+1)\frac{S_{k-1}}w-kuS_k.
\end{align}
Note that from \eqref{curkl::1}
\begin{align}
    C_n^lD_sS_k=C_n^kD_sS_l,
\end{align}
we have
\begin{align}
    &C_n^l\big((S_k^{ij}\big(\frac{u_i}w\big)_sx_su)_j+(S_k^{si}\frac{x_s}{w})_i-(n-k+1)\frac{S_{k-1}}w-kuS_k\big)\nonumber\\
    =&C_n^k\big((S_l^{ij}\big(\frac{u_i}w\big)_sx_su)_j+(S_l^{si}\frac{x_s}{w})_i-(n-l+1)\frac{S_{l-1}}w-luS_l\big)
\end{align}
Integrate it on $\Omega$ and use \eqref{curkl::1}, we obtain the identity we want.
\end{proof}
\begin{proof}[Proof of Theorem \ref{mainthm3}.]
By Lemma \ref{lemma::curkl::12} and \ref{lemma::curkl::14}, we obtain
\begin{align}
    (n-k+1)C_n^l\int_\Omega S_{k-1}\big(\frac1{\sqrt 2}-\frac1w\big)\mathrm dx-(n-l+1)C_n^k\int_\Omega S_{l-1}\big(\frac1{\sqrt 2}-\frac1w\big)\mathrm dx=(k-l)C_n^l\int_\Omega uS_k\mathrm dx.
\end{align}
That is 
\begin{align}
    (n-l+1)\int_\Omega S_{k-1}\big(\frac1{\sqrt 2}-\frac1w\big)\big(1-\frac{(n-l+1)C_n^kS_{l-1}}{(n-k+1)C_n^lS_{k-1}}\big)\mathrm dx=(k-l)\int_\Omega uS_k.
\end{align}
By  \eqref{gNMieq} and \eqref{mainpb::curkl}, we have
\begin{align}
    \bigg(\frac{S_{k-1}/C_n^{k-1}}{S_{l-1}/C_n^{l-1}}\bigg)^\frac1{k-l}\geq \bigg(\frac{S_{k}/C_n^k}{S_{l}/C_n^l}\bigg)^\frac1{k-l}=1\geq \frac{S_k/C_n^k}{S_{k-1}/C_n^{k-1}}.
\end{align}
So 
\begin{align}
    \frac{S_{l-1}}{S_{k-1}}\frac{(n-l+1)C_n^k}{(n-k+1)C_n^l}\leq \frac lk\quad\text{and}\quad S_k\leq \frac{n-k+1}kS_{k-1}.
\end{align}
Note that by Lemma \ref{lemma::curkl::13}, we have 
\begin{align}
    \frac1{\sqrt2}-\frac1w\leq u\quad\text{in }\Omega, 
\end{align}
So
\begin{align}
    &(n-k+1)\int_\Omega S_{k-1}\big(\frac1{\sqrt2}-\frac1w\big)\Big(1-\frac{(n-l+1)C_n^kS_{l-1}}{(n-k+1)C_n^lS_{k-1}}\Big)V\mathrm dx\nonumber\\
    \leq&\frac{(n-k+1)(k-l)}k\int_\Omega S_{k-1}uV\mathrm dx
\end{align}
On the other hand,
\begin{align}
    (k-l)\int_\Omega S_kuV\mathrm dx\geq \frac{2(n-k+1)(k-l)}k\int_\Omega S_{k-1}uV\mathrm dx.
\end{align}
So we have
\begin{align}
    \frac{n-k+1}kS_{k-1}=S_k.
\end{align}
By \eqref{NMieq}, we infer the eigenvalues of $D(\frac{Du}w)$ are all equal to $1$. Combine with the boundary conditions in \eqref{mainpb3}, we obtain 
\begin{align}
    u=-\sqrt{1-|x|^2}+\frac1{\sqrt 2},
\end{align}
and $\Omega$ is a ball with radius $\frac1{\sqrt 2}$.
\end{proof}

%%%% Acknowledgments %%%%%%%%
\section*{Acknowledgments}
The authors would like to thank Prof. Xi-Nan Ma for his discussions and advice. The research is supported by the National Science Foundation of China No. 11721101 and the National Key R and D Program of China 2020YFA0713100.

\bibliographystyle{plain}
\bibliography{GJZ20220913}

\end{document}